\documentclass[reqno,11pt]{amsart}

\usepackage{mathrsfs}
\usepackage{amssymb,amscd,amsthm, amsfonts,lipsum,amsmath}
\usepackage{times}
\usepackage{flafter}
\usepackage{graphicx}
\usepackage{mathtools}
\usepackage{etoolbox}
\usepackage{geometry}
\usepackage{enumerate}
\usepackage{amscd}
\usepackage{amsthm}
\usepackage{amsmath, nccmath}
\usepackage{
	amsmath,  
	amssymb,  
	amsthm,   
	tikz,     
	fullpage, 
	ytableau, 
	thmtools, 
	hyperref, 
	cite,     
	url       
}
\usepackage{mathtools}
\newtheorem{definition}{Definition}[section]
\newtheorem{example}{Example}[section]
\newtheorem{remark}{Remark}[section]
\newtheorem{corollary}{Corollary}[section]
\newtheorem{theorem}{Theorem}[section]
\newtheorem{proposition}{Proposition}[section]
\newtheorem{lemma}{Lemma}[section]

\def\multiset#1#2{\ensuremath{\left(\kern-.3em\left(\genfrac{}{}{0pt}{}{#1}{#2}\right)\kern-.3em\right)}}

\numberwithin{equation}{section}

\begin{document}
	\title[Integral of natural powers of the numerical values with application to combinatorics and invariant matrix norms.]{ New applications to combinatorics and invariant matrix norms of an integral
 representation of natural powers of the numerical values	}
	\date{May 23, 2021}
	\author[H. Issa]{Hassan Issa$^{*\dag\ddag}$}
	\author[B. Mourad]{Bassam Mourad$^*$}
	\author[H. Abbass]{Hassan Abbas$^*$}

\address{*Department of Mathematics, Faculty of Science I,
	Lebanese University, Beirut, Lebanon.}
\address{ $^\dag$ first  author: Hassan A. Issa}
\email[$^\ddag$ corresponding author]{hissa@uni-math.gwdg.de}
	\noindent
	\begin{abstract}
		Let $\vee^k A$ be the $k$-th symmetric tensor power of $A\in M_n(\mathbb{C})$.  In \cite{IAM}, we have expressed the normalized trace of  $\vee^kA$  as an integral of the $k$-th powers of the numerical values of $A$ over the unit sphere $\mathbb{S}^{n}$ of $\mathbb{C}^{n}$ with respect to the normalized Euclidean surface measure $\sigma$.	In this paper, we first use this integral representation to construct a family of unitarily invariant norms on $ M_n(\mathbb{C})$ and then explore their relations to Schatten-norms of  $\vee^k A$. Another application yields a connection between the analysis of  symmetric gauge functions with that  of complete symmetric polynomials.  Finally, motivated by the work of R. Bhatia and J. Holbrook in \cite{hol},  and as pointed out by R. Bhatia in \cite{bhatia} in the development of the theory of weakly unitarily invariant  norms, we provide  an explicit form  for the weakly unitarily invariant norm corresponding to the $L^4$-norm on the space $C(\mathbb{S}^{n})$ of continuous functions on the sphere. Our result generalize those of R. Bhatia and J. Holbrook in different directions and pave the way to  a technique for computing those weakly unitarily invariant  norms on $ M_n(\mathbb{C})$ that are associated to  $L^{2k}$-norms on $C(\mathbb{S}^{n})$.
	\end{abstract}
	\maketitle
	\bigskip
	\section{Introduction}

   For $A\in M_n(\mathbb{C})$, let $trA$ be its trace and $\vee^k A$ denote its $k$-th symmetric tensor product where $k\in\mathbb{N}$. In addition, let  $Tr(\vee^kA)=\dfrac{tr(\vee^kA)}{c_{n,k}}$ be the normalized trace of $\vee^k A$ with $c_{n,k}=\binom{n+k-1}{k}$.  In \cite{IAM}, we have expressed  $Tr(\vee^kA)$  as an integral of the $k$-th powers of the numerical values of $A$ over the unit sphere $\mathbb{S}^{n}$ of $\mathbb{C}^{n}$ with respect to the normalized Euclidean surface measure $\sigma$.	
   More precisely, we have proved the following.
	
	\begin{equation}
	\int_{\mathbb{S}^{n}}\big(<A\xi,\xi>\big)^kd\sigma(\xi)=Tr(\vee^k(A))=\frac{1}{c_{n,k}}\sum_{l=1}^k\sum_{\beta\in S^k_l}\frac{1}{z_\beta}\prod_{t=1}^l tr(A^{\beta_t}),\label{dr21}
	\end{equation}
	where $$S^k_l=\{\beta=(\beta_1,\beta_2,\cdots,\beta_k)\in\mathbb{N}^k\mid \beta_1\geq\beta_2\geq\cdots\geq\beta_k ,|\beta|=k, \beta_l\neq0 \mbox{ and } \beta_i=0  \mbox{ for all } i>l\}.$$

Since $Tr(\vee^kA)$ is simply the evaluation of the normalized complete $k$-homogenous polynomials $H_k$ at the eigenvalues of $A$, the preceding integral can be then interpreted as a representation for such polynomials.

	In this paper, we  continue to explore the potential consequences of (1.1) in many directions. For convenience, we shall use the same  terminology and notation  as in \cite{IAM} and we shall always assume that  $k$ is  non-zero.

Recall that the Hilbert-Schmidt norm $\lVert\cdot\rVert_{(2)}$ of  the integral transform $T_{A,k}$ on $L^2(\mathbb{S}^{n})$ satisfies
	\begin{equation}\label{dr5}
	c^\frac{1}{2k}_{n,k}\lVert T_{A,k}\rVert^\frac{1}{k}_{(2)}=\left(\int_{\mathbb{S}^{n}}\lVert A\xi\rVert_2^{2k}d\sigma(\xi)\right)^\frac{1}{2k}.
	\end{equation}
	From the measure theory point of view, this can be regarded as the  $L^{2k}(\mathbb{S}^n,\mathbb{C}^n)$ Bochner norm of the map $\mathbb{S}^n\ni\xi\longrightarrow A\xi\in\mathbb{C}^n$, whenever  $\mathbb{C}^n$ is equipped with $\lVert \cdot\rVert_2$.  A slightly careful argument leads to a norm on $M_n(\mathbb{C})$ but with some extremal property of unitary invariance. Note that, a norm $\tau$ on $M_n(\mathbb{C})$ is called unitarily invariant (u.i. for short) if it satisfies
	\begin{equation}\label{dr6}
	\tau(A)=\tau(UAV), \quad\mbox{for all}\quad A\in M_n(\mathbb{C})\quad\mbox{and all} \quad U,V\in U(n),
	\end{equation}
	where $U(n)$ is the group of all unitary matrices in $M_n(\mathbb{C})$.

A detailed description of this paper is presented as follows.
In Section 2, we shall take advantage of  Thompson triangular inequality to prove that for each $k\in\mathbb{N}$ the right hand side of (\ref{dr5}) defines a Q-norm  i.e. a norm for which there exists a unitarily invariant norm $N_k$ on $M_n(\mathbb{C})$ such  $\sqrt{N_k(A^\star A)}$ is given by (\ref{dr5}) (cf. \cite{b1,b2}). Using (1.1), we provide an expression for $N_k$  as a non-trivial combination of Schatten p-norms $\lVert \cdot\rVert_{(p)}$ on $M_n(\mathbb{C})$. A connection to a Schatten norm	of the symmetric power is also explored.
	
Characterized by Von-Neumann in \cite{neu},  unitarily invariant norms are special norms on the sequence of singular values and this characterization is clarified via the concept of symmetric gauge functions. Recall that a symmetric gauge function (s.g.f..) is a norm on $\mathbb{R}^n$ being invariant under the permutation and the change of signs of the coordinates. Of special interest are the inequalities satisfied by these functions (cf. \cite{bhatia} Ch. IV and the references therein).  In Section 3, we apply such inequalities to explicitly construct those unitarily invariant norms studied in Section 2. This yields various simple proofs for the analysis of the complete symmetric polynomials $h_k$, and this approach is apparently novel in combinatorics and has the advantage of giving a new look at the topic from a different perspective.

The last section is devoted to  the investigation of a problem on weakly unitarily invariant (w.u.i.) norms which dates back to the work of R. Bhatia and J. Holbrook in \cite{hol} in 1987.
 Recall that a norm $\tau$ on $M_n(\mathbb{C})$ is called weakly unitarily invariant norm if it satisfies the invariant property
	\begin{equation}\label{dr7}
	\tau(A)=\tau(UAU^\star), \quad\mbox{for all}\quad A\in M_n(\mathbb{C})\quad\mbox{and all} \quad U\in U(n),
	\end{equation}
	where $U^\star$ denotes as usual the conjugate transpose of $U$. Such norms  that have been studied by C.-K.Fong and J.A.R. Holbrook in \cite{fong} on arbitrary dimensional Hilbert spaces, were characterized by Chi-K. Li  and  N. K. Tsing in \cite{Li} as Schur-convex norms (rather than symmetric gauge functions for u.i. norms) when considered as maps on the cone of Hermitian matrices $\mathcal{H}_n$. However, R. Bhatia and  J.A.R. Holbrook (\cite{hol}) provided a complete characterization for w.u.i. norms in terms of "unitary function norms" on the space $C(\mathbb{S}^{n})$ of continuous functions on the sphere. Following \cite{hol}, a  norm $\Phi$ on  $C(\mathbb{S}^{n})$ is called unitarily invariant  if
	\begin{equation}\label{dr8}
	\Phi(f)=\tau(f\circ U), \quad\mbox{for all}\quad f\in C(\mathbb{S}^{n})\quad\mbox{and all} \quad U\in U(n).
	\end{equation}
	Given a  norm $\Phi$ on  $C(\mathbb{S}^{n})$, let $\Phi':M_n(\mathbb{C})\longrightarrow \mathbb{R}$ be the map defined by
	\begin{equation}\label{dr9}
	\Phi'(A)=\Phi(f_A),\quad \mbox{ where } f_A:\mathbb{S}^n\longrightarrow\mathbb{C} \quad \mbox{is given by: } f_A(\xi)=<A\xi,\xi>.
	\end{equation}
	More precisely, their characterization  is given as follows: A norm $\tau$ on $M_n(\mathbb{C})$ is w.u.i if and only if there is a unitarily invariant norm $\Phi$ on  $C(\mathbb{S}^{n})$ such that $\tau=\Phi'$  (cf.   \cite{hol}, Theorem 2.1).

	This explicit correspondence between $\tau$ and $\Phi$ could lead to a more development of the theory of w.u.i. norms and such correspondence is "not known for even much used norms" (\cite{bhatia}, p. 110). In particular, this holds for the natural
	$L^p$-norms $\Phi_p$  on  $C(\mathbb{S}^{n})$ (w.r.t $\sigma$) with $p\in[1,\infty[\backslash\{2\}$. It is direct to see that $\Phi'_\infty$ is the numerical radius norm. The case when $p=2$, was obtained by R. Bhatia and J. Holbrook  (cf.  \cite{hol}, Theorem 2.4) and explicitly it reads as
	\begin{equation}\label{dr10}
	\Phi'_2(A)=\Bigg(\frac{\lVert A\rVert_F^2+|tr(A)|^2}{c_{n,2}}\Bigg)^\frac{1}{2}.
	\end{equation}

	We shall make use of (1.1)  to present the exact values for the following:
	\begin{enumerate}
		\item  $\Phi'_k$ with  $k\in\mathbb{N}$, on the cone of positive semi-definite matrices $\mathcal{H}_n^+$,
		\item  $\Phi'_{2k}$ on  $\mathcal{H}_n$,
		\item  $\Phi'_{4}$ on $M_n(\mathbb{C})$.
	\end{enumerate}
 Furthermore, we shall modify $\Phi'_{4}$ in order to obtain a new w.u.i. norm with some  probabilistic interpretation. Our result generalizes the inner product form on $M_n(\mathbb{C})$ obtained in \cite{hol}. Finally. we conclude with a method  that produces the exact value for $\Phi'_{2k}$ on $M_n(\mathbb{C})$. 
	\section{On a class of unitarily invariant norms}
	In this section, we shall first provide a collection of unitarily invariant norms on $M_n(\mathbb{C})$ indexed by real numbers $q\geq1$. In the case where $q$ is an integer, we shall show that these norms are formed of a non-trivial combination of products  of Schatten $p$-norms. Then, we shall  focus on a family of Q-norms of the form (\ref{dr5}) which will be used later in our investigation for the weakly unitarily invariant norms in the last section.

In the sequel, for $A\in M_n(\mathbb{C})$ and $p\in[1,\infty[$ we shall write $|A|:=\sqrt{A^\star A}$ for the positive part of $A$,  $\lambda:=\lambda(A)$ for the unordered tuple of eigenvalues of $A$, $\sigma(A):=\lambda(|A|)\in \mathbb{R}^n_+$ and $\lVert A\rVert_{(p)}:=\big[tr(|A|^p)\big]^\frac{1}{p}$ for the Schatten $p$-norms. Moreover, to indicate that $A\in\mathcal{H}_n^+$ we simply write $A\geq0$ and as usual for  $A,B\geq0$ we write $A\geq B$ if $A-B\geq0.$
	\begin{definition}
		Given $q\in[1,\infty[$ and $n\in\mathbb{N}$, we define the map $N_q:M_n(\mathbb{C})\longrightarrow\mathbb{R}_+$  by
		\begin{equation}\label{ar1} N_q(A):=\Bigg(\int_{\mathbb{S}^{n}}\Big\lVert |A|^\frac{1}{2}\xi\Big\rVert_2^{2q}d\sigma(\xi)\Bigg)^\frac{1}{q},
		\end{equation}
		where $\lVert\cdot\rVert_2$ denotes the Euclidean norm on $\mathbb{C}^n.$
	\end{definition}
	We aim to show that each $N_q$ defines a unitarily invariant norm on $M_n(\mathbb{C})$. Even though, Formula (\ref{ar1}) can be viewed as the $L^q(\mathbb{S}^n)$-norm for the map $\xi\longrightarrow\lVert |A|^\frac{1}{2}\xi\rVert_2^{2}$, this is not sufficient to obtain Minkowski's inequality. When $q=k$ is an integer,  then by (\ref{dr21}) we can write
	\begin{equation}\label{ar2}
	N_k(A)=\Big[Tr\big(\vee^k |A|\big)\Big]^\frac{1}{k}, \quad A\in M_n(\mathbb{C}).
	\end{equation}
	Now the convexity of the map $A\longrightarrow\big[Tr(\vee^k(A))\big]^\frac{1}{k}$ on the set of positive matrices $\mathcal{H}_n$ certainly ensures that $N_k(A+B)\leq N_k(A)+N_k(B)$ whenever $A,B\in\mathcal{H}^+_n$ (cf. \cite{marcus,mcleod} and Ch. II in \cite{bhatia}). Such approach is thus limited to $\mathcal{H}^+_n$. It is worthy to mention here that one may combine the triangular inequality of Ky Fan norms,  the monotonicity and the convexity of the $k$th root of the complete symmetric polynomials to solve such a problem for the integer case.  However, we will prove Minkowski's inequality for $N_q$ for any $q\geq1$  using Thompson triangular inequality together with the triangular inequality  of the $L^q(\mathbb{S}^n)$-norm but without use of any other convexity results in this discipline. Let us recall  Thompson triangular inequality for the positive parts of matrices.
	\begin{lemma}[ \cite{thompson},  Theorem 2 ]\label{lar1} Let $A,B\in M_n(\mathbb{C})$ then there are $U,V\in U(n)$ satisfying
		\begin{equation}\label{ar3}
		|A+B|\leq U|A|U^\star+V|B|V^\star.
		\end{equation}
	\end{lemma}
	With the above notation, we now present our first result of this section.
	\begin{theorem}\label{pt1}
		For each $q\geq1$ and each $n\in\mathbb{N}$,  the map $N_q$ given by (\ref{ar1}) defines a unitarily invariant norm on $M_n(\mathbb{C})$.
	\end{theorem}
	\begin{proof}
		Let us first notice that for any $C\in M_n(\mathbb{C})$  and every $U,V\in U(n)$, the  following identity holds
		\begin{align}\label{ar4}
		\int_{\mathbb{S}^{n}}\Big\lVert UCV\xi\Big\rVert_2^{2q}d\sigma(\xi)=\int_{\mathbb{S}^{n}}\Big\lVert C\xi\Big\rVert_2^{2q}d\sigma(\xi).
		\end{align}
		Indeed,
		\begin{equation*}
		\begin{split}
		\int_{\mathbb{S}^{n}}\Big\lVert UCV\xi\Big\rVert_2^{2q}d\sigma(\xi)&=\int_{\mathbb{S}^{n}}\Big\lVert CV\xi\Big\rVert_2^{2q}d\sigma(\xi)\\
		&=\int_{\mathbb{S}^{n}}\big( <CV\xi,CV\xi>\big)^qd\sigma(\xi)=\int_{\mathbb{S}^{n}}\big( <C\mu,C\mu>\big)^qd\sigma(\mu),
		\end{split}
		\end{equation*}
		where the first equality follows from the basic fact that unitary matrices preserves the Euclidean norm and the last equality holds by the invariance of $\sigma$ under a unitary transformation. The first two conditions for $N_q$ to be a norm on $M_n(\mathbb{C})$ are easy to obtain. For the triangular inequality, we fix $A,B\in M_n(\mathbb{C})$ then by Lemma \ref{lar1} there are unitary matrices $U$ and $V$ satisfying (\ref{ar3}). So that for any $\xi\in\mathbb{S}^{n}$ we have
		\begin{align*}
		\Big\lVert |A+B|^\frac{1}{2}\xi\Big\rVert_2^2&=<|A+B|\xi,\xi>\\
		& \leq   <U|A|U^\star\xi,\xi>+<V|B|V^\star\xi,\xi>\\
		&=\Big\lVert|A|^\frac{1}{2}U^\star\xi\Big\rVert_2^2+\Big\lVert|B|^\frac{1}{2}V^\star\xi\Big\rVert_2^2
		\end{align*}
		Considering the  $L^q(\mathbb{S}^{n})$-norm for the functions on both sides of the preceding inequality  yields
		
		\begin{align}
		\Bigg(\int_{\mathbb{S}^{n}}\Big\lVert |A+B|^\frac{1}{2}\xi\Big\rVert_2^{2q}d\sigma(\xi)\Bigg)^\frac{1}{q}&\leq
		\left(\int_{\mathbb{S}^{n}}\Bigg(\Big\lVert |A|^\frac{1}{2}U^\star \xi\Big\rVert_2^{2}+\Big\lVert |B|^\frac{1}{2}V^\star\xi\Big\rVert_2^2\Bigg)^qd\sigma(\xi)\right)^\frac{1}{q}\notag\\
		&\leq
		\Bigg(\int_{\mathbb{S}^{n}}\Big\lVert |A|^\frac{1}{2}U^\star \xi\Big\rVert_2^{2q}\Bigg)^\frac{1}{q}+\Bigg(\int_{\mathbb{S}^{n}}\Big\lVert |B|^\frac{1}{2}V^\star\xi\Big\rVert_2^{2q}d\sigma(\xi)\Bigg)^\frac{1}{q}\label{ar5}\\
		&=\Bigg(\int_{\mathbb{S}^{n}}\Big\lVert |A|^\frac{1}{2}\xi\Big\rVert_2^{2q}d\sigma(\xi)\Bigg)^\frac{1}{q}+\Bigg(\int_{\mathbb{S}^{n}}\Big\lVert |B|^\frac{1}{2}\xi\Big\rVert_2^{2q}d\sigma(\xi)\Bigg)^\frac{1}{q},\label{ar6}
		\end{align}
		where (\ref{ar5}) follows from the triangular inequality of the $L^q(\mathbb{S}^{n})$-norm and (\ref{ar6}) follows by suitable applications of (\ref{ar4}). Next, to check the unitary invariance property (\ref{dr6}) of $N_q$, consider $A\in M_n(\mathbb{C})$  and  $U,V\in U(n)$ then
		\begin{align}
		N_q(UAV)&=\Bigg(\int_{\mathbb{S}^{n}}\Big\lVert |UAV|^\frac{1}{2}\xi\Big\rVert_2^{2q}d\sigma(\xi)\Bigg)^\frac{1}{q}
		=\Bigg(\int_{\mathbb{S}^{n}}\Big\lVert |AV|^\frac{1}{2}\xi\Big\rVert_2^{2q}d\sigma(\xi)\Bigg)^\frac{1}{q}\label{ar7}\\
		&=\Bigg(\int_{\mathbb{S}^{n}}\Big\lVert V^\star|A|^\frac{1}{2}V\xi\Big\rVert_2^{2q}d\sigma(\xi)\Bigg)^\frac{1}{q}=\Bigg(\int_{\mathbb{S}^{n}}\Big\lVert |A|^\frac{1}{2}\xi\Big\rVert_2^{2q}d\sigma(\xi)\Bigg)^\frac{1}{q},\label{ar8}
		\end{align}
		where (\ref{ar7}) follows from the fact that $|UC|=|C|$ and  the first equality of (\ref{ar8}) is obtained by making use  of the fact that $|CV|=V^\star|C|V$  for any $C\in M_n(\mathbb{C})$. While the second equality of (\ref{ar8}) follows easily from (\ref{ar4}).
	\end{proof}
For the case $q=k\in\mathbb{N}$, we  replace $A$ by $|A|$ in the second equality  of (\ref{dr21}) and we obtain an expression for $N_k$ in terms of Schatten norms as follows.
	\begin{corollary}\label{car1}
		Let $A\in M_n(\mathbb{C})$ and let $k\in\mathbb{N}$ then the following holds
		\begin{equation}\label{ar0}
		c_{n,k}N_k^k(A)=\sum_{l=1}^k\sum_{\beta\in S^k_l}\frac{1}{z_\beta}\prod_{t=1}^l \big\lVert A\big\rVert^{\beta_t}_{_{(\beta_t)}}.
		\end{equation}
	\end{corollary}
	 Motivated by the above result  and similar to Example 2.1 in \cite{IAM}, we express  $N_k$ in terms of Schatten norms for the cases $k=3,4$ and $5$ (the cases $k=1$ or $k=2$ can be obtained directly from the results in \cite{hol}).
	\begin{example}\label{edr2} For a given	 $k,n\in\mathbb{N}$ and for any  $A\in M_n(\mathbb{C})$ the following identities hold
		\begin{fleqn}[\parindent]
			\begin{equation*}
			\begin{split}
			\mbox{i)}\quad	c_{n,3}N_3^3(A)
			=&\frac{1}{3}\lVert A\rVert^3_{(3)}+\frac{1}{2}\lVert A\rVert^2_{(2)}.\lVert A\rVert_{(1)}+\frac{1}{6}\lVert A\rVert^3_{(1)}.\\
			\mbox{ii)}\quad	c_{n,4}N_4^4(A)
			=&\frac{1}{4}\lVert A\rVert^4_{(4)}+\frac{1}{3}\lVert A\rVert^3_{(3)}\lVert A\rVert_{(1)}+\frac{1}{8}\lVert A\rVert^4_{(4)}+\frac{1}{4}\lVert A\rVert^2_{(2)}\lVert A\rVert^2_{(1)}+\frac{1}{24}\lVert A\rVert^4_{(1)}.\\
			\mbox{iii)}\quad c_{n,5}N_5^5(A)
			=&\frac{1}{5}\lVert A\rVert^5_{(5)}+\frac{1}{4}\lVert A\rVert^4_{(4)}\lVert A\rVert_{(1)}+\frac{1}{6}\lVert A\rVert^3_{(3)}\lVert A\rVert^2_{(2)}+\frac{1}{6}\lVert A\rVert^3_{(3)}\lVert A\rVert^2_{(1)}+\frac{1}{8}\lVert A\rVert^4_{(2)}\lVert A\rVert_{(1)}\\
			&+\frac{1}{12}\lVert A\rVert^2_{(2)}\lVert A\rVert^3_{(1)}+\frac{1}{120}\lVert A\rVert^5_{(1)}.
			\end{split}
			\end{equation*}
		\end{fleqn}	
	\end{example}
	
	Making use of Minkowski's inequality for gauge functions (Theorem IV.1.8 in\cite{bhatia}), one can construct new unitarily invariant norms from a given u.i. norm. In particular, we may apply such a construction on each $N_q$ to obtain a more general u.i. norm. Let $q\geq1$ be fixed, for each $p\in[1,\infty[$ consider the map $N_q^{(p)}:M_n(\mathbb{C})\longrightarrow\mathbb{R}_+$ defined by
	\begin{equation}\label{ar9}
	N_q^{(p)}(A):=N^\frac{1}{p}_q(|A|^p)=\Bigg(\int_{\mathbb{S}^{n}}\Big\lVert |A|^\frac{p}{2}\xi\Big\rVert_2^{2q}d\sigma(\xi)\Bigg)^\frac{1}{pq}
	=\Bigg(\int_{\mathbb{S}^{n}}\Big(<|A|^p\xi,\xi>\Big)^{q}d\sigma(\xi)\Bigg)^\frac{1}{pq}.
	\end{equation}
	Then $N_q^{(p)}$ defines a u.i. norm on $M_n(\mathbb{C})$ (cf. for example Ex. IV.2.7 in \cite{bhatia}). On the one hand, using a similar argument as
the one used in the previous discussion, we can conclude that if $q=k$ is an integer then $N_k^{(p)}$ satisfies  the following identity
	\begin{equation}\label{ar10}c_{n,k}\big[N_k^{(p)}(A)\big]^{pk}=\sum_{l=1}^k\sum_{\beta\in S^k_l}\frac{1}{z_\beta}\prod_{t=1}^l \big\lVert A\big\rVert^{p\beta_t}_{_{(p\beta_t)}}, \quad p\in[1,\infty[, A\in M_n(\mathbb{C}).\end{equation}
	On the other hand, by formulating the previous equation using the first equality in (\ref{dr21}),  one obtains the "surprising result" that the Schatten p-norms (with a modified power) on the space of the $k$th-symmetric powers of $M_n(\mathbb{C})$ provides a u.i. norm on $M_n(\mathbb{C})$. The following proposition demonstrates this fact.
	\begin{proposition}\label{pdr3}
		Let $k\in\mathbb{N}$ and $p\in[1,\infty[$, then
		\begin{equation}\label{ar11}c^\frac{1}{kp}_{n,k}N_k^{(p)}(A)=\Big\lVert\vee^k A\Big\rVert_{(p)}^\frac{1}{k}, \quad  A\in M_n(\mathbb{C}).\end{equation}
	\end{proposition}
	\begin{proof}
		From (\ref{dr21}) we know that $$c_{n,k}\big[N_k^{(p)}(A)\big]^{pk}=c_{n,k}\int_{\mathbb{S}^{n}}\Big(<|A|^p\xi,\xi>\Big)^{k}d\sigma(\xi)=tr\Big(\vee^k |A|^p\Big).$$
		Thus, it suffices to prove that \begin{equation}\label{ar13}tr\Big(\vee^k |A|^p\Big)=tr\Big(\big|\vee^kA\big|^p\Big).\end{equation} Let us show first that for any $B\geq0$ and any $\alpha\in\mathbb{R}_+$ the following holds
		\begin{equation}\label{ar12}tr\Big(\vee^k (B^\alpha)\Big)=tr\Big((\vee^kB)^\alpha\Big).\end{equation}
		The equality $\vee^k (B^q)=(\vee^kB)^q$ for $q\in\mathbb{Q}_+$ follows from the fact that $\vee^k B\vee^k B=\vee^k B^2$ and by a simple induction.  Now for $\alpha\in\mathbb{R}_+$, choose a sequence $(q_m)\subset \mathbb{Q}_+$ converging to $\alpha$. By the continuity of $\lambda$ as a map from $M_n(\mathbb{C})$ onto $\mathbb{C}^n$, we conclude  that the sequence of vectors $$\lambda(B^{q_m})=\big(\lambda_1(B^{q_m}),\cdots,\lambda_n(B^{q_m})\big)=\Big(\big(\lambda_1(B)\big)^{q_m},\cdots,\big(\lambda_n(B)\big)^{q_m}\Big)$$
		converges to
		$$\lambda(B^\alpha)=\Big(\big(\lambda_1(B)\big)^\alpha,\cdots,\big(\lambda_n(B)\big)^\alpha\big).$$
		Now if $f_k$ denotes the $k$th complete symmetric polynomial on $\mathbb{C}^n$, which is trivially continuous,  then
		\begin{align*}
		tr\Big(\vee^k (B^\alpha)\Big)=f_k(\lambda(B^\alpha))&=\lim_{m\rightarrow\infty}f_k\lambda(B^{q_m})=\lim_{m\rightarrow\infty}tr\big(\vee^k (B^{q_m})\big)\\
		&=\lim_{m\rightarrow\infty}tr\Big((\vee^kB)^{q_m}\Big)=tr\Big((\vee^kB)^\alpha\Big),
		\end{align*}
		where the last equality follows from the continuity of the trace on $M_{c_{n,k}}(\mathbb{C})$. Now using the basic fact that  $\vee^k|A|=|\vee^kA|$, Eq. (\ref{ar13}) then follows by applying (\ref{ar12}) to the matrix $B=|A|$.
	\end{proof}
	For the case $p=2$, we obtain a family of Q-norms $\{N_q^{(2)}\}_{q\geq1}$ (see Def. IV 2.9 in \cite{bhatia}). Those norms  will be used in Section 3 for the analysis of the complete homogeneous  and in Section 4 to obtain some upper estimations for a family of w.u.i. norms. For simplicity, we shall write $N'_q$ rather than $N_k^{(2)}$ i.e.
	\begin{equation}\label{4r1}
	N'_q(A)=\Bigg(\int_{\mathbb{S}^{n}}\Big\lVert |A|\xi\Big\rVert_2^{2q}d\sigma(\xi)\Bigg)^\frac{1}{2q}=\Bigg(\int_{\mathbb{S}^{n}}\lVert A\xi\rVert_2^{2q}d\sigma(\xi)\Bigg)^\frac{1}{2q}
	.\end{equation}
	We conclude this section by presenting a proof for (\ref{dr5}). Following the notation in Section 1, we shall make use of the following well known formula (cf. Theorem 3.5 in \cite{zhu}) to compute the Hilbert-Schmidt norm of $T_{A,k}$:
	$$\big\lVert T_{A,k}\big\rVert^2_{(2)}=\int_{\mathbb{S}^{n}}\int_{\mathbb{S}^{n}}\big|<A\xi,\mu>\big|^{2k}d\sigma(\xi)d\sigma(\mu).$$
	With this in mind, applying Fubini's theorem to the preceding equation, Eq. (\ref{dr5}) now follows easily from the following lemma.
	\begin{lemma}
		For any $z\in\mathbb{C}^n$, the following formula holds
		\begin{equation}\label{4r3}c_{n,k}\int_{\mathbb{S}^{n}}\big|<z,\mu>\big|^{2k}d\sigma(\mu)=\big\lVert z\big\rVert_2^{2k}.\end{equation}
	\end{lemma}
	\begin{proof}
		In view of the multinomial theorem, we can write
		$$\big|<z,\mu>\big|^{2k}=\Big|\big(<z,\mu>\big)^k\Big|^2=\Big|\sum_{\substack{|\alpha|=k \\ \alpha\in \mathbb{N}^n_0 }}\binom{k}{\alpha}z^\alpha \ \overline{\mu}^\alpha\Big|^2=\sum_{\substack{|\alpha|=|\beta|=k \\ \alpha,\beta\in \mathbb{N}^n_0 }}\binom{k}{\alpha}\binom{k}{\beta}z^\alpha \ \overline{z}^\beta\mu^\beta \ \overline{\mu}^\alpha.$$
		Integrating both sides   of the preceding equality w.r.t. $\mu$ over $\mathbb{S}^{n}$, and using Lemma 2.1 in \cite{IAM} we get
		$$\int_{\mathbb{S}^{n}}\big|<z,\mu>\big|^{2k}d\sigma(\mu)=\sum_{\substack{|\alpha|=k \\ \alpha\in \mathbb{N}^n_0 }}\binom{k}{\alpha}\binom{k}{\alpha}\dfrac{(n-1)!\alpha!}{(n-1+|\alpha|)!}z^\alpha \ \overline{z}^\alpha
		=\frac{k!(n-1)!}{(n-1+k)!}\sum_{\substack{|\alpha|=k \\ \alpha\in \mathbb{N}^n_0}}\binom{k}{\alpha}z^\alpha \ \overline{z}^\alpha=\frac{\big\lVert z\big\rVert_2^{2k}}{c_{n,k}}.$$
	\end{proof}
\section{Complete homogeneous symmetric polynomials and gauge functions}
The main purpose of this section is to apply our results concerning the unitarily invariant norms and that of Q-matrix norms to the study of the q-complete homogeneous symmetric functions as introduced in \cite{IAM}.

Recall that in \cite{IAM}, we pointed out that for any $A\in M_n(\mathbb{C})$ Eq. (\ref{dr21}) can be reformulated as
 \begin{equation}\label{xfx5}
\int_{\mathbb{S}^{n}}\big(<A\xi,\xi>\big)^kd\sigma(\xi)=H_k(\lambda),
\end{equation}
where $\lambda=\lambda(A)$ and $H_k$ is the normalized complete $k$-homogeneous polynomial in $n$ complex variables given by
$$H_k(z)=\frac{k!(n-1)!}{(n+k-1)!}\sum_{\substack{|\alpha|=k \\ \alpha\in \mathbb{N}^n}} 	z^\alpha, \quad z\in\mathbb{C}^n.$$
In particular, for every $x\in\mathbb{R}^n$
\begin{equation}\label{hsn}
H_k(x)=\int_{\mathbb{S}^{n}}\big(<X\xi,\xi>\big)^kd\sigma(\xi),
\end{equation}
where $X:=\texttt{diag}(x_1,\cdots,x_n)$. Also, we need  from \cite{IAM} the following case of the generalization for the restriction of $H_k$ onto $\mathbb{R}^n_+$. Given $q\geq1$ the $q$-homogeneous function $H_q$ is defined $\mathbb{R}^n_+$ by
\begin{equation}\label{hsn}
H_q(x)=\int_{\mathbb{S}^{n}}\big(<X\xi,\xi>\big)^qd\sigma(\xi),
\end{equation}
Let us first introduce some more notation. Given $x,y\in\mathbb{R}^n$ and $p\geq0$, we write
$$|x|=(|x_1|,\cdots,|x_n|), \quad |x|^p=(|x_1|^p,\cdots,|x_n|^p),\quad xy=(x_1y_1,\cdots,x_ny_n)\quad\mbox{ and}\quad x^\downarrow=(x_{[1]},\cdots,x_{[n]}),$$
where $x_{[i]}$ denotes the ith component obtained from the components of $x$ after rearranging them in decreasing order. We also  use the usual weak majorization notation i.e.
 $$x\prec_w y \quad\mbox{if}\quad \sum^k_{i=1} x_{[i]}\leq\sum^j_{i=1} y_{[i]}, \ \  j=1\cdots n$$
 If $x\prec_w y$  and $\sum^n_{i=1} =x_i\leq\sum^n_{i=1} y_i$, we say that $x$ is mjorized by $y$ and we write $x\preccurlyeq y$.

Motivated by our construction of the unitarily invariant norms and the Q-norms obtained in Section 2, we now state our first result for this section.
\begin{theorem}\label{pt2}
For each $q\geq1$,    the map $\Phi_q$ defined on $\mathbb{R}^n$ by
	$$\Phi_q(x):=H_q^\frac{1}{q}(|x|),$$
is a symmetric gauge function. 
\end{theorem}
\begin{proof}
By Theorem \ref{pt1}, we know that
$$N_q(A)=\Bigg(\int_{\mathbb{S}^{n}}\Big\lVert |A|^\frac{1}{2}\xi\Big\rVert_2^{2q}d\sigma(\xi)\Bigg)^\frac{1}{q}=\Bigg(\int_{\mathbb{S}^{n}}\big(<|A|\xi,\xi>\big)^qd\sigma(\xi)\Bigg)^\frac{1}{q}$$
defines a unitarily invariant norm on $M_n(\mathbb{C})$. Hence it  defines a s.g.f. (see, for example, \cite{bhatia}) denoted by $\Phi_q$, on $\mathbb{R}^n$ given by
$$\Phi_q(x)=N_q(X),$$
where $X:=\texttt{diag}(x_1,\cdots,x_n)$. From (\ref{hsn}) we obtain
$$N_q(X)=\Bigg(\int_{\mathbb{S}^{n}}\big(<|X|\xi,\xi>\big)^qd\sigma(\xi)\Bigg)^\frac{1}{q}
=H^\frac{1}{q}_q(x).$$
This shows that $\Phi_q=H^\frac{1}{q}_q$ and the proof is complete.
\end{proof}
It is worth mentioning here that even in the case where $q=k\in\mathbb{N}$, and to the best of our knowledge,  the connection of the analysis between  the homogeneous symmetric polynomials and the  symmetric gauge functions  has not been pointed out before. However, one could obtain this result (for $q=k$) from  the literature by combining the facts that $H_k$ is $k$-homogeneous, positive definite on $\mathbb{R}^n_+$ and  invariant under permutation of coordinates with the convexity results obtained by  Marcus and Lopes \cite{marcus} (cf. pages 116 and 119 in\cite{Marshall}).

As a result, we have the following conclusion.

\begin{corollary}
Given two real numbers $q\geq1$ and $p\geq1$. Then, for any $x,y\in\mathbb{R}^n$ we have
		\begin{equation}\label{n1}
	H_q(|xy|)\leq H^\frac{1}{2}_q(x^2)H^\frac{1}{2}_q(y^2)
		\end{equation}
		and
\begin{equation}\label{n2}
H^\frac{1}{qp}_q(|x+y|^p)\leq H^\frac{1}{qp}_k(|x|^p)+H^\frac{1}{qp}_q(|y|^p).
\end{equation}
		\end{corollary}
\begin{proof}
	It suffices to see that the  two inequalities are H\"{o}lder and Minkowski's inequality for symmetric gauge functions (cf. for example  \cite{bhatia}, p: 88) when applied to $H^\frac{1}{q}_q$.
\end{proof}
Equation (\ref{n1})  is a generalization for the multiplicatively convex (cf. \cite{nic}) property for $H_k$ obtained by Y.-M Chu et al.  in 2011 (cf. \cite{chu}). Moreover, Eq. (\ref{n2}) is a generalization for  Theorem 3.9 in \cite{sra}, obtained by S. Sra in 2020, where a discrete version for fractional power of Minkowski's inequality was needed.

  Taking advantage of the conclusion of Problem IV . 5 . 2 . in \cite{bhatia}, we  obtain the following result which to the best of our knowledge is new even for the case $q=k\in\mathbb{N}$.
\begin{corollary}
Let $q\geq1$ and $p\in]0,1[$. Then, for any $x,y\in\mathbb{R}^n$ we have
\begin{equation}\label{n4}
H^\frac{1}{pq}_q(|x+y|^p)\leq2^{\frac{1}{p}-1}\Big[ H^\frac{1}{pq}_q(|x|^p)+H^\frac{1}{pq}_q(|y|^p)\Big].
	\end{equation}
\end{corollary}
 Recall that every s.g.f. $\Phi$ is (convex and monotone) on $\mathbb{R}^n_+$ (cf. for example, \cite{bhatia},  p. 45) hence it is strongly isotone i.e. given $x,y\in\mathbb{R}^n_+$ then
 $$x\preccurlyeq_w y\Longrightarrow \Phi(x)\leq\Phi(y).$$
 Thus,  we obtain the following.
 \begin{corollary}
 	Let $q\geq1$ then for any $x,y\in\mathbb{R}^n_+$ we have
 	\begin{equation}\label{n5}
x\preccurlyeq_w y\Longrightarrow 	H_q(x)\leq H_q(y).
 	\end{equation}
 \end{corollary}
\begin{proof}
	From the previous discussion, we know that $H^\frac{1}{q}_q$ is strongly isotone on $\mathbb{R}^n_+$. Therefore,
	 $$x\preccurlyeq_w y\Longrightarrow H^\frac{1}{q}_q(x)\leq H^\frac{1}{q}_q(y)\Longrightarrow H_q(x)\leq H_q(y), \quad x,y\in\mathbb{R}^n_+.$$
\end{proof}
 It is worthy to note here that Eq. (\ref{n5}) is a generalization for the work of Guan in \cite{guan}. Furthermore, motivated by the work of D. B. Hunter in \cite{hunter}, T. Tao presented (in his blog) a proof for the Schur-convexity of the even degree complete symmetric polynomials on $\mathbb{R}^n$.
 His proof was based on applying some differential operators to $H_{2k}$.

  Now in connection with this,  we are in a position  to present a simpler proof for the Schur convexity property and to extend it to  $\mathbb{R}^n$. For this reason, we shall need a characterization  for w.u.i. norms which was obtained by Chi-K. Li and N. K. Tsing in \cite{Li}.
\begin{lemma}[\cite{Li}, Theorem 4.1]
A map $N:\mathcal{H}_n\longrightarrow\mathbb{R}$ is (the restriction of) a w.u.i norm if and only if there is a Schur convex norm $\Phi:\mathbb{R}^n\longrightarrow\mathbb{R}$ such that
\begin{equation}\label{n7}
N(A)=\Phi(\lambda(A)), \quad A\in\mathcal{H}_n.
\end{equation}
\end{lemma}
Here is our result generalizing the Schur-convexity for $H_{2k}$ on $\mathbb{R}^n$.
\begin{corollary}
For each $q\geq1$, consider the map $\mathcal{H}_{2q}$ defined on $\mathbb{R}^n$ by
$$\mathcal{H}_{2q}(x):=\bigg[\int_{\mathbb{S}^{n}}\big(<X\xi,\xi>\big)^2\bigg]^{q}d\sigma(\xi),$$
where $X=\texttt{diag}(x_1,\cdots,x_n)$. Then, $\mathcal{H}_{2q}$ is Schur convex on $\mathbb{R}^n$.
\end{corollary}
\begin{proof}
From Theorem 2.1 in \cite{hol}, we know that
$$\Phi'_{2q}(A)=\Big(\int_{\mathbb{S}^{n}}\Big|<A\xi,\xi>\Big|^{2q}d\sigma(\xi)\Big)^\frac{1}{2q}$$
is a w.u.i. norm on $M_n(\mathbb{C})$. Given $x\in\mathbb{R}^n$ which obviously means that $X\in\mathcal{H}$ so that
$$\Phi'_{2q}(X)=\Big(\int_{\mathbb{S}^{n}}\Big|<X\xi,\xi>\Big|^{2q}d\sigma(\xi)\Big)^\frac{1}{2q}=\Big(\int_{\mathbb{S}^{n}}\bigg[\big(<X\xi,\xi>\big)^2\bigg]^{q}d\sigma(\xi)\Big)^\frac{1}{2q}=\mathcal{H}^{\frac{1}{2q}}_{2q}(x).$$
Applying the preceding lemma to $N=\Phi'_{2q}$, we obtain that $\mathcal{H}^{\frac{1}{2q}}_{2q}$ is a Schur convex (norm) i.e.
 $$x\prec y\Longrightarrow \mathcal{H}^{\frac{1}{2q}}_{2q}(x)\leq \mathcal{H}_{2q}(y)\Longrightarrow \mathcal{H}_{2q}(x)\leq \mathcal{H}^{\frac{1}{2q}}_{2q}(y), \quad x,y\in\mathbb{R}^n.$$
\end{proof}
We conclude this section by providing an interpolating sequence of inequalities for $H_q$ on $\mathbb{R}^n_+$ as follows.
\begin{theorem}
	For any pair of real numbers $q\geq p\geq1$, the following formula holds
	$$H_q(x^p)\leq H_p(x^q), \quad x\in \mathbb{R}^n_+.$$
	\end{theorem}
\begin{proof}
	From McCarty inequality, we know that
	$$\Big(\big<X^p\xi,\xi\big>\Big)^\frac{q}{p}\leq \big<(X^p)^\frac{q}{p}\xi,\xi\big>=\big<X^q\xi,\xi\big>,$$
for each $\xi\in\mathbb{S}^n$. The result now follows easily by taking the $p$th power for both sides of the above inequality and then integrating.
\end{proof}
\begin{remark}
	One may use Theorem 3.1 to obtain other generalization for similar inequalities involving  $q$-homogeneous symmetric functions. We shall point out a particular application in this direction. Following \cite{b3}, given two s.g.f. $f$ and $g$ one can construct a new s.g.f. $f\div g$ defined on the set of  decreasing tuples $\mathbb{R}^n_\downarrow$ by
	$$f\div g(x):=\max\Big\{\frac{f(yx)}{g(y)} \ \  : \  y\in\mathbb{R}^n_\downarrow\backslash\{0\}\Big\}$$
and then extend it by symmetry to $\mathbb{R}^n$. Now let $q,q'\in[1,\infty[$ and consider $f:=H_q^\frac{1}{q}$ and $g:=H_{q'}^\frac{1}{q'}$ then for $x\in\mathbb{R}^n_\downarrow$ the function
$$H_{q,q'}(x):=\max\Big\{\frac{H_q^\frac{1}{q} \big(|y||x|\big)}{H_{q'}^\frac{1}{q'}\big(|y|\big)}   \ \  : \ y\in\mathbb{R}^n_\downarrow\backslash\{0\}\Big\}$$
extends to a s.g.f. 	on $\mathbb{R}^n$. In particular, if $q=q'$ then by (\ref{n1}) it follows that $H_{q,q'}=H^{\frac{1}{q}}_q$. As $H_{q,q'}$ is a s.g.f. then such generalization for $q$-homogeneous functions satisfies Equations (\ref{n1})-(\ref{n5}).

	\end{remark}
	\section{Explicit value for a class of w.u.i. norms induced by $L^k$-norms  }
	This section is devoted to continue the work of R. Bhatia and J.R. Holbrook in \cite{hol} concerning giving expressions for the w.u.i. norms (on $M_n(\mathbb{C})$) associated to the natural $L^p$ norms on $C(\mathbb{S}^n)$. In the next theorem, we combine some of the results obtained in \cite{hol} to suit our needs.	
	\begin{theorem}[\cite{hol}]
		Let $A,B\in M_n(\mathbb{C}))$ then
		\begin{equation}\label{aa0}
		\int_{\mathbb{S}^{n}}<A\xi,\xi><B\xi,\xi>d\sigma(\xi)=\frac{1}{c_{n,2}}\Big(tr(AB)+tr(A)tr(B)\Big).\end{equation}
		In particular, \begin{equation}\label{aa1}
		\Phi'^2_2(A):=\int_{\mathbb{S}^{n}}\Big|<A\xi,\xi>\Big|^2d\sigma(\xi)=\frac{\lVert A\rVert_F^2+|tr(A)|^2}{c_{n,2}}.
		\end{equation}
		Moreover, if $E$ denotes the expectation and $V$ is the variance (w.r.t. $d\sigma$) then the Frobenius norm satisfies
		\begin{equation}\label{aa2}
		\big\lVert A\big\rVert^2_{(2)}=nE\big(|f_A|^2\big)+n^2V(f_A),\end{equation}
		where $f_A$ is defined by (\ref{dr9}).
	\end{theorem}
	The authors proved (\ref{aa0}) using some techniques that are based on Riesz representation theorem that deals with characterizing all w.u.i. sesquilinear forms on $M_n(\mathbb{C})$. Taking $B=I_n$ in (\ref{aa0}), we clearly obtain Eq. (\ref{dr21}) for $k=1$.  Choosing $B=A^\star$, (Eq. \ref{dr21})  follows easily and finally using (\ref{aa1} and (\ref{dr21}) Eq. (\ref{aa2}) is obtained. The argument on the sesquilinear forms which is used in \cite{hol} is not sufficient to examine the explicit value of
	\begin{equation}\label{aa4}
	\Phi'_k(A):=\Big(\int_{\mathbb{S}^{n}}\Big|<A\xi,\xi>\Big|^kd\sigma(\xi)\Big)^\frac{1}{k},  \ \  A\in M_n(\mathbb{C}),\end{equation}
	for general $k\in\mathbb{N}$ (cf. Section IV.  in \cite{bhatia} for a motivation to this problem).   So far, we have obtained the following results:
	\begin{fleqn}[\parindent]
		\begin{equation}\label{aa4}
		\Phi'_k(A)=\Big[Tr(\vee^kA)\Big]^\frac{1}{k}=N_k(A)=\Big[\frac{1}{c_{n,k}}\sum_{l=1}^k\sum_{\beta\in S^k_l}\frac{1}{z_\beta}\prod_{t=1}^l \big\lVert A\big\rVert^{\beta_t}_{_{(\beta_t)}}\Big]^\frac{1}{k}, \quad A\in\mathcal{H}_n^+
		\end{equation}
		\begin{equation}\label{aa5}
		\Phi'_{2k}(A)=\Big[Tr(\vee^{2k}A)\Big]^\frac{1}{2k}=N_{2k}(A)=\Big[\frac{1}{c_{n,2k}}\sum_{l=1}^{2k}\sum_{\beta\in S^{2k}_l}\frac{1}{z_\beta}\prod_{t=1}^l \big\lVert A\big\rVert^{\beta_t}_{_{(\beta_t)}}\Big]^\frac{1}{2k}, \quad A\in\mathcal{H}_n.
		\end{equation}
	\end{fleqn}
	Our next result deals with presenting an upper estimation for   $\Phi'_{2p}$ using the Q-norms defined by (\ref{4r1})
\begin{proposition}
	Let $A\in M_n(\mathbb{C})$ and let $p\geq1$ then
	\begin{equation}\label{aa6}\Phi'^{2p}_{2p}(A)\leq\frac{1}{2}\Big(\Phi'^p_p(A^2)+N'^{2p}_p(A)\Big).
	\end{equation}
\end{proposition}
\begin{proof}
	Recall that for any $A\in M_n(\mathbb{C})$ we have
	$$\Big|<A\xi,\xi>\Big|^{2p}\leq \frac{1}{2}\Bigg(\Big|<A^2\xi,\xi>\Big|^p+\Big\lVert A\xi\Big\rVert^p\Big\lVert A^\star\xi\Big\rVert^p\Bigg).$$
	Integrating both sides of the preceding inequality leads to
	\begin{align}
	\Phi'^{2p}_{2p}(A)&\leq\frac{1}{2}\Big(\Phi'^p_p(A^2)+\int_{\mathbb{S}^{n}}\Big\lVert A\xi\Big\rVert^p\Big\lVert A^\star\xi\Big\rVert^pd\sigma(\xi)\Big)\notag\\
	&\leq\frac{1}{2}\Bigg(\Phi'^p_p(A^2)+\Big[\int_{\mathbb{S}^{n}}\Big\lVert A\xi\Big\rVert^{2p}d\sigma(\xi)\Big]^\frac{1}{2}\Big[\int_{\mathbb{S}^{n}}\Big\lVert A^\star\xi\Big\rVert^{2p}d\sigma(\xi)\Big]^\frac{1}{2}\Bigg)\label{aa7}\\
	&=\frac{1}{2}\Big(\Phi'^p_p(A^2)+N'^{2p}_p(A)\Big),\label{aa8}
	\end{align}
	where (\ref{aa7}) is the usual H\"{o}lder inequality and (\ref{aa8}) follows from the fact that $N'_p$ is unitarly invariant.
\end{proof}
	The next theorem is a generalization  of (\ref{aa0}). Its proof uses Lemma \ref{laa}, which for convenience, is given  in the appendix below.	
	\begin{theorem}\label{tm0}
		Let $A, B, C, D \in M_n(\mathbb{C})$ be arbitrary matrices and put $c_n:=n(n+1)(n+2)(n+3)$. Consider $\sigma$  the normalized probability measure on the unit sphere $\mathbb{S}^n$  of  $\mathbb{C}^n$.  Then,
		\begin{fleqn}[\parindent]
			\begin{fleqn}[\parindent]
				\begin{align}\label{aa00}
				&c_n\int_{\mathbb{S}^{n}}<A\xi,\xi><B\xi,\xi><C\xi,\xi><D\xi,\xi>d\sigma(\xi)=\notag\\
				&tr\Big[ABCD+ABDC+ACBD+ACDB+ADBC+ADCB\Big] +tr(A)tr \Big[BCD+BDC\Big]\notag\\
				&+tr(B)tr \Big[ACD+ADC\Big]+tr(C)tr \Big[ABD+ADB\Big]+tr(D)tr \Big[ABC+ACB\Big]\notag\\
				&+tr(AB)\Big[tr (CD)+tr(C)tr(D)\Big]+tr(AC)\Big[tr (BD)+tr(B)tr(D)\Big]\notag\\
				&+tr(AD)\Big[tr (BC)+tr(B)tr(C)\Big]+tr(A)\Big[tr(B)tr (CD)+tr(C)tr (BD)+tr(D)tr (BC)\Big]\notag\\
				&+tr(A)tr(B)tr(C)tr(D).
				\end{align}
			\end{fleqn}
		\end{fleqn}
	\end{theorem}
	\begin{proof}
		From (\ref{dr21}), we know that:
		\begin{equation}\label{i1}
		c_n\int_{\mathbb{S}^{n}}\big(<A\xi,\xi>\big)^4d\sigma(\xi)=6tr(A^4)+8[tr(A^3)].tr(A)+3[tr (A^2)]^2+6[tr (A^2)](tr A)^2+[tr (A)]^4.
		\end{equation}
		Let $x,y,z$ be formal variables and put
		\begin{equation}\label{aa01}
		P(x,y,z)=A+xB+yC+zD\in M_n(\mathbb{C}[x,y,z]).
		\end{equation}
		Applying (\ref{i1}) to $P(x,y,z)$ yields the following equality for multinomials in $\mathbb{C}[x,y,z]$
		\begin{align}\label{i2}
		c_n\int_{\mathbb{S}^{n}}\big(<P(x,y,z)\xi,\xi>\big)^4d\sigma(\xi)=&6tr\big[P^4(x,y,z)\big]+8tr\big[P^3(x,y,z)\big].tr\big[P(x,y,z)\big]\notag\\
		&+3\bigg[tr \big[P^2(x,y,z)\big]\bigg]^2\notag+6tr \big[P^2(x,y,z)\big]\cdot\bigg[tr \big[P(x,y,z)\big]\bigg]^2\\
		&+\bigg[tr \big[P(x,y,z)\big]\bigg]^4.
		\end{align}
		Applying the multinomial theorem to the integrand in the above equality yields
		\begin{fleqn}[\parindent]
			\begin{equation*}\label{aa001}
			\begin{split}
			&\int_{\mathbb{S}^{n}}\Big(<P(x,y,z)\xi,\xi>\Big)^4d\sigma(\xi)\\
			&=\int_{\mathbb{S}^{n}}\Bigg(<A\xi,\xi>+x<B\xi,\xi>+y<C\xi,\xi>+z<D\xi,\xi>\Bigg)^4d\sigma(\xi)\\
			&=\int_{\mathbb{S}^{n}}\sum_{\substack{|\alpha|=4 \\ \alpha\in\mathbb{N}_0^4}}\binom{4}{\alpha}\Big(<A\xi,\xi>,x<B\xi,\xi>,y<C\xi,\xi>,z<D\xi,\xi>\Big)^\alpha d\sigma(\xi)\\
			&=\sum_{\substack{\alpha_1+\cdots+\alpha_4=4 \\ \alpha\in\mathbb{N}_0^4}}\binom{4}{\alpha}\int_{\mathbb{S}^{n}}\big(<A\xi,\xi>\big)^{\alpha_1}\big(x<B\xi,\xi>\big)^{\alpha_2}\big(y<C\xi,\xi>\big)^{\alpha_3}
\big(z<D\xi,\xi>\big)^{\alpha_4}d\sigma(\xi).
			\end{split}
			\end{equation*}
		\end{fleqn}
		In particular, the $xyz$-variable corresponds to $\alpha=(1,1,1,1)$ i.e. the coefficient of the $xyz$-variable on the left side of (\ref{i2}) is given by
		\begin{equation}\label{i0}
		24c_n\int_{\mathbb{S}^{n}}<A\xi,\xi><B\xi,\xi><C\xi,\xi><D\xi,\xi>d\sigma(\xi).
		\end{equation}
		The coefficient of the $xyz$-variable on the right hand  side of (\ref{i2}) is given by 24 multiplied by the right hand side of (\ref{aa00}) and then  the theorem follows form Lemma \ref{laa}.
	\end{proof}
	The following corollaries are direct applications of the preceding theorem. Recalling that  $\Re(.)$ denotes the real part and setting first $D=I_n$ in (\ref{aa00}) we obtain the following conclusion.
	\begin{corollary}
		Let $A, B, C\in M_n(\mathbb{C})$  and put $d_n:=6c_{n,3}=n(n+1)(n+2)$ then
		\begin{align}\label{i69}
		d_n\int_{\mathbb{S}^{n}}<A\xi,\xi><B\xi,\xi><C\xi,\xi>d\sigma(\xi)=tr&\Big[ABC+ACB\Big]+tr(A)tr(BC)+tr(B)tr(AC)\notag\\
		&+tr(C)tr(AB)+tr(A)tr(B)tr(C).
		\end{align}
		In particular, if $C\in\mathcal{H}_n^+$ and $d\mu(\xi):=<C\xi,\xi>d\sigma(\xi)$ then
		\begin{equation}\label{i70}
		d_n\int_{\mathbb{S}^{n}}\Big|<A\xi,\xi>\Big|^2d\mu(\xi)=tr\Big[AA^\star C+ACA^\star \Big]
		+2\Re\Big[tr(A)tr(A^\star C)\Big]+tr(C)\lVert A\rVert^2_F+\Big|tr(A)\Big|^2tr(C).\end{equation}
	\end{corollary}
	Note that the preceding corollary generalizes (\ref{aa0}) and(\ref{aa1}). Now let $B=A$ and $C=D=A^\star$ this time, we get
	\begin{corollary} \label{cp1}
		Let $A\in M_n(\mathbb{C})$  and  $c_n:=n(n+1)(n+2)(n+3)$, then
		\begin{align}\label{i110}
		c_n\int_{\mathbb{S}^{n}}\Big|<Au,u>\Big|^4d\sigma(u)&=4tr\big(A^2A^{\star 2} \big) +2tr\Big[\big( A^\star A\big)^2\Big]+2\Big[tr\big(A^\star A\big)\Big]^2+8\Re\Big[ tr\big(A\big)tr\big(AA^{\star 2}\big)\Big]\notag\\
		&+2\Re \Bigg[\big[tr(A^2)\big]\big[trA^\star\big]^2\Bigg]+\Big|tr\big(A^2\big)\Big|^2+4\Big[tr\big(A^\star A\big)\Big]\cdot\Big|tr(A)\Big|^2+\Big| tr(A)\Big|^4.\notag
		\end{align}
		Equivalently	$\Phi'_4(A)$ is given by
		\begin{equation}\label{ccc1}
		\resizebox{1.01\hsize}{!}{$\left(\frac{4\big\lVert A^2\big\rVert_F^2+2\big\lVert A\big\rVert^4_{(4)}+2\big\lVert A\big\rVert_F^4+8\Re\Big[ tr\big(A\big)tr\big(AA^{\star 2}\big)\Big]
				+2\Re \Bigg[\big[tr(A^2)\big]\big[trA^\star\big]^2\Bigg]+4\big\lVert A\big\rVert_F^2\Big|tr(A)\Big|^2+\Big|tr\big(A^2\big)\Big|^2+\Big| tr(A)\Big|^4}{c_n}\right)^\frac{1}{4}.$}
		\end{equation}
	\end{corollary}
	\begin{remark}
		In the connection to the study for estimations of the w.u.i. norms associated to the $L^p$-norms on $C(\mathbb{S}^n)$,
		one can find an upper estimation for $\Phi'_p$, whenever $p\in]2,4[$ in-terms $\Phi'_2$ and $\Phi'_4$. By a direct application of Riesz-Thorin interpolation theorem (cf. Theorem 2.5 in\cite{zhu}) we obtain  	
		$$\Phi'_p(A)\leq \big(\Phi'_2(A)\big)^{\frac{4}{p}-1}\big(\Phi'_4(A)\big)^{2-\frac{4}{p}}, \quad\mbox{for any} \  A\in M_n(\mathbb{C}),$$
		where $\Phi'_2$ and $\Phi'_4$ are given (\ref{aa1}) and (\ref{ccc1}), respectively.
	\end{remark}
	The expression for $\Phi'_4$  involves many terms, even if we restrict our attention to normal matrices.  This motivates us to  slightly modify the $\Phi'_4$-norm  in order to obtain some (other) weakly unitarily invariant norms   with less numbers of terms in (\ref{ccc1}). This will also lead to some probabilistic interpretation for a certain combination of the Schatten 2 and 4-norms.
	Let $T:C(\mathbb{S}^n)\longrightarrow C(\mathbb{S}^n)$ be the linear map defined by
	\begin{equation}\label{qeq1}
	T(f)=f-\int_{\mathbb{S}^{n}}f.
	\end{equation}
	In addition, let $\varPsi_1$ and $\varPsi_2$ be two unitarily invariant norms on    $C(\mathbb{S}^n)$ and consider
	\begin{equation}\label{i71}
	\varPsi(f)=\Big(\varPsi_1^4(f)+\varPsi_2^4(Tf)\Big)^\frac{1}{4}, \quad f \in C(\mathbb{S}^n).
	\end{equation}
	Notice that,  $T(f\circ U)=(Tf)\circ U$ for every $f \in C(\mathbb{S}^n)$ and each $U\in U_n.$ Therefore,  $\varPsi$ is a unitarily invariant norm on $C(\mathbb{S}^n)$.
	\begin{theorem}\label{tp1}
		Let $\varPsi$ be the unitarily invariant norm on    $C(\mathbb{S}^n)$ defined by (\ref{i71}), where  $$\varPsi_1(\cdot)=\sqrt[4]{3d_n}\Big\lVert\cdot\Big\rVert_{L^4}\quad \mbox{and}\quad \varPsi_2(\cdot)=\sqrt[4]{nd_n}\Big\lVert\cdot\Big\rVert_{L^4}.$$
		Then the weakly unitarily invariant norm on $M_n(\mathbb{C})$  induced by $\varPsi$, denoted by $N_\varPsi$,  satisfies
		\begin{align}\label{i72}N^4_\varPsi(A)=&4\lVert A^2\rVert_F^2+2\lVert A\rVert^4_{(4)}+2\lVert A\rVert_F^4+
		\frac{4}{n}\Re \Bigg[\big[tr(A^2)\big]\big[trA^\star\big]^2\Bigg]+\frac{8}{n}\lVert A\rVert_F^2\Big|tr(A)\Big|^2\notag\\
		&+\Big|tr\big(A^2\big)\Big|^2+\frac{3n-6}{n^2}\Big| tr(A)\Big|^4.\end{align}
	\end{theorem}
	\begin{proof}
		Let  $\gamma=\frac{-tr(A)}{n}$ and put $f=f_A$. Integrating the identity, $|f+\gamma|^4=|f|^4+|\gamma|^4+4|\gamma|^2|f|^2+2\Re \overline{\gamma}^2f^2+4(|f|^2+|\gamma|^2)\Re \overline{\gamma}f$,
		\quad leads to
		\begin{fleqn}[\parindent]\begin{align}
			\int_{\mathbb{S}^{n}}|f+\gamma|^4=&\big(\Phi'_4(A)\big)^4+|\gamma|^4+4|\gamma|^2\big(\Phi'_2(A)\big)^2+\frac{2}{n(n+1)}\Re \Big[\overline{\gamma}^2\big(tr(A^2)+n^2\gamma^2\big)\Big]\notag\\
			&+4\Re \Big[\overline{\gamma}\big(\int_{\mathbb{S}^{n}}|f|^2f+|\gamma|^2\int_{\mathbb{S}^{n}}f\big)\Big]\label{i80}\\
			&=\big(\Phi'_4(A)\big)^4+|\gamma|^4+4|\gamma|^2\big(\Phi'_2(A)\big)^2+\frac{2}{n(n+1)}\Big[n^2|\gamma|^4+\Re \big(\overline{\gamma}^2tr(A^2)\big)\Big]-4|\gamma|^4\notag\\
			&+\frac{4}{d_n}\Re \Bigg[\overline{\gamma}\Big(2tr(A^2A^\star)
			+2tr(A)tr(A^\star A)+tr(A^\star)tr(A^2)+\Big|tr(A)\Big|^2tr(A)\Big)\Bigg]\label{i81}\\
			&=\big(\Phi'_4(A)\big)^4+|\gamma|^4+4|\gamma|^2\big(\Phi'_2(A)\big)^2+\frac{2}{n(n+1)}\Big[n^2|\gamma|^4+\Re \big(\overline{\gamma}^2tr(A^2)\big)\Big]-4|\gamma|^4\notag\\
			&+\frac{4}{d_n}\Bigg[\Re \Big(2\overline{\gamma}tr(A^2A^\star)
			-n\overline{\gamma}^2tr(A^2)\Big)-2n|\gamma|^2\lVert A\rVert_F^2-n^3|\gamma|^4\Bigg]\notag\\
			&=\big(\Phi'_4(A)\big)^4+4|\gamma|^2\Big[\big(\Phi'_2(A)\big)^2-\frac{2n}{d_n}\lVert A\rVert_F^2\Big]+\Big(\frac{2}{n(n+1)}-\frac{4n}{d_n}\Big)\Re \overline{\gamma}^2tr(A^2)\notag\\
			&+\frac{8}{d_n}\Re \Big(\overline{\gamma}tr(A^2A^\star)\Big)+(-3+\frac{2n}{n+1}-\frac{4n^3}{d_n})|\gamma|^4\notag\\
			&=\big(\Phi'_4(A)\big)^4+4|\gamma|^2\Big[\big(\Phi'_2(A)\big)^2-\frac{2n}{d_n}\lVert A\rVert_F^2\Big]+\frac{4-2n}{d_n}\Re \overline{\gamma}^2tr(A^2)+\frac{8}{d_n}\Re \Big(\overline{\gamma}tr(A^2A^\star)\Big)\notag\\
			&-\frac{5n^3+5n^2+6n}{d_n}|\gamma|^4.\notag
			\end{align}\end{fleqn}
		where (\ref{i80}) follows from (\ref{aa0}) and (\ref{aa1}), whereas  (\ref{i81}) follows from (\ref{i69}). Using the preceding equality, one gets $\varPsi_2^4(Tf)=nd_n\int_{\mathbb{S}^{n}}|f+\gamma|^4$ is given by
		\begin{fleqn}[\parindent]
			\begin{equation*}
			\resizebox{0.96\hsize}{!}{$nd_n\big(\Phi'_4(A)\big)^4+(8n-4n^2)|\gamma|^2\lVert A\rVert_F^2+(4n-2n^2)\Re \overline{\gamma}^2tr(A^2)+8n\Re \Big(\overline{\gamma}tr(A^2A^\star)\Big)+(-n^4+3n^3-6n^2)|\gamma|^4.$}\end{equation*}\end{fleqn}
		By Eq. (\ref{ccc1}), we obtain that $N^4_\varPsi(A)=3d_n\int_{\mathbb{S}^{n}}\Big|f\Big|^4+nd_n\int_{\mathbb{S}^{n}}|f+\gamma|^4$ is given by
		\begin{align*}
		N^4_\varPsi(A)=&c_n\big(\Phi'_4(A)\big)^4+(8n-4n^2)|\gamma|^2\lVert A\rVert_F^2
		+(4n-2n^2)\Re \overline{\gamma}^2tr(A^2)+8n\Re \Big(\overline{\gamma}tr(A^2A^\star)\Big)\\
		&+(-n^4+3n^3-6n^2)|\gamma|^4\\
		=&4\lVert A^2\rVert_F^2+2\lVert A\rVert^4_{(4)}+2\lVert A\rVert_F^4+8\Re\Big[ tr\big(A\big)tr\big(AA^{\star 2}\big)\Big]
		+2\Re \Bigg[\big[tr(A^2)\big]\big[trA^\star\big]^2\Bigg]\\
		&+4\lVert A\rVert_F^2\Big|tr(A)\Big|^2+\Big|tr\big(A^2\big)\Big|^2
		+\Big| tr(A)\Big|^4 +(\frac{8}{n}-4)\Big|tr(A)\Big|^2\lVert A\rVert_F^2+(\frac{4}{n}-2)\Re \big[trA^\star\big]^2tr(A^2)\\
		&-8\Re \Big(tr\big(A^\star\big)tr(A^2A^\star)\Big)+(-1+\frac{3}{n}-\frac{6}{n^2})\Big| tr(A)\Big|^4\\
		=&4\lVert A^2\rVert_F^2+2\lVert A\rVert^4_{(4)}+2\lVert A\rVert_F^4+
		\frac{4}{n}\Re \Bigg[\big[tr(A^2)\big]\big[trA^\star\big]^2\Bigg]
		+\frac{8}{n}\lVert A\rVert_F^2\Big|tr(A)\Big|^2+\Big|tr\big(A^2\big)\Big|^2\\
		&+\frac{3n-6}{n^2}\Big| tr(A)\Big|^4.
		\end{align*}
	\end{proof}
	A similar method to the above technique yields a  reduction for the terms in (\ref{i72}). Indeed, it is enough to  consider the following unitarily invariant norm on $C(\mathbb{S}^n)$
	\begin{equation}\label{i75}
	\varPsi_0(f)=\big(\varPsi^4(f)+\varPsi_3^4(Tf)\big)^\frac{1}{4}, \quad f \in C(\mathbb{S}^n),
	\end{equation}
	where $\varPsi$ is the norm considered in the previous theorem and $\varPsi_3(\cdot):=\sqrt{2n(n+1)}\Big\lVert\cdot\Big\rVert_{L^2(C(\mathbb{S}^n))}$. 
	\begin{theorem}\label{tp1}
		Let $\varPsi_0$ be the unitarily invariant norm on    $C(\mathbb{S}^n)$ defined by (\ref{i75}).	The weakly unitarily invariant norm on $M_n(\mathbb{C})$  induced by $\varPsi_0$, denoted by $N_{\varPsi_0}$,  satisfies
		\begin{equation}\label{i76}N^4_{\varPsi_0}(A)=4\lVert A^2\rVert_F^2+2\lVert A\rVert^4_{(4)}+6\lVert A\rVert_F^4+
		\frac{4}{n}\Re \Bigg[\big[tr(A^2)\big]\big[trA^\star\big]^2\Bigg]+\Big|tr\big(A^2\big)\Big|^2+\frac{3n-2}{n^2}\Big| tr(A)\Big|^4.\end{equation}
	\end{theorem}
	\begin{proof}
		Let $f$ be as in the previous proof, then
		\begin{align*}
		\Bigg[\int_{\mathbb{S}^{n}}\Big|f-\int_{\mathbb{S}^{n}} f\Big|^2\Bigg]^2&=\Bigg[\int_{\mathbb{S}^{n}}|f|^2-\Big|\int_{\mathbb{S}^{n}} f\Big|^2\Bigg]^2
		=\Bigg[\big(\Phi'_2(A)\big)^2-\frac{\Big|tr\big(A\big)\Big|^2}{n^2}\Bigg]^2\\
		&=\Bigg[\dfrac{n\lVert A\rVert_F^2-\Big|tr\big(A\big)\Big|^2}{n^2(n+1)}\Bigg]^2
		=\frac{1}{n^4(n+1)^2}\Bigg[n^2\lVert A\rVert_F^4-2n\lVert A\rVert_F^2\Big|tr\big(A\big)\Big|^2+\Big|tr\big(A\big)\Big|^4\Bigg].
		\end{align*}
		Hence,
		$$\varPsi_3^4(Tf)=4n^2(n+1)^2\Bigg[\int_{\mathbb{S}^{n}}\Big|f-\int_{\mathbb{S}^{n}} f\Big|^2\Bigg]^2=4\lVert A\rVert_F^4-\frac{8}{n}\lVert A\rVert_F^2\Big|tr\big(A\big)\Big|^2+\frac{4}{n^2}\Big|tr\big(A\big)\Big|^4$$
		and (\ref{i76}) follows by adding the preceding equation to ((\ref{i72})).
	\end{proof}
	The preceding two theorems can be combined together to provide another weakly unitarily invariant norm. We state this result in a probabilistic language.
	\begin{corollary}
		Given $A\in M_n(\mathbb{C})$, consider the random variable $f_A(\xi)=<A\xi,\xi>$ on $(\mathbb{S}^n,d\sigma)$. For any $\alpha\geq0$, the following identity holds
		\begin{align*}
		&3d_nE(|f_A|^4)+nd_n\mu_4^4(f_A)+4\alpha n^2(n+1)^2V^2(f_A)=\\
		&4\lVert A^2\rVert_F^2+2\lVert A\rVert^4_{(4)}+(2+4\alpha)\lVert A\rVert_F^4+
		\frac{4}{n}\Re \Bigg[\big[tr(A^2)\big]\big[trA^\star\big]^2\Bigg]
		+\frac{8}{n}(1-\alpha)\lVert A\rVert_F^2\Big|tr(A)\Big|^2+\Big|tr\big(A^2\big)\Big|^2\\
		&+\frac{3n-6+4\alpha}{n^2}\Big| tr(A)\Big|^4,
		\end{align*}
		where $E$ is the expectation, $V$ is the variance and $\mu_4$ is the fourth central moment.	
	\end{corollary}
	Motivated by the results of this section, we provide a method for finding an expression of  $\Phi'_{2k}$ on $M_n(\mathbb{C})$ and for arbitrary $k\in\mathbb{N}$. we shall first start by exploring a method to find an expression of 	
	\begin{equation}\label{fgh}
	\int_{\mathbb{S}^{n}}\prod_{i=1}^k<A_i\xi,\xi>d\sigma(\xi),
	\end{equation}
	for arbitrary $A_1,\cdots,A_k\in M_n(\mathbb{C})$. In order to do that, let us consider one of the representations obtained in Section 2 say
	\begin{equation}\label{dr00012}
	\int_{\mathbb{S}^{n}}\big(<A\xi,\xi>\big)^kd\sigma(\xi)=Tr(\vee^kA).
	\end{equation}
	Let $x_2,x_3,\cdots,x_k$ be formal variables and for convenience put $x_1=1$, then replace $A$ in the preceding equation by  the following matrix	
	$$B=A_1+x_2A_2+x_3A_3+\cdots+x_kA_k.$$
	which has multinomial entries, we obtain
	\begin{align}
	\int_{\mathbb{S}^{n}}\big(<B\xi,\xi>\big)^kd\sigma(\xi)&=\int_{\mathbb{S}^{n}}\Big(\sum_{i=1}^kx_i<A_i\xi,\xi>\Big)^kd\sigma(\xi)\notag\\
	& =\sum_{\substack{|\alpha|=k \\ \alpha\in\mathbb{N}_0^k}}\binom{k}{\alpha}x^\alpha\int_{\mathbb{S}^{n}}\prod_{i=1}^k(<A_i\xi,\xi>)^{\alpha_i}d\sigma(\xi)\label{kg}\\
	&=Tr\Bigg(\vee^k\Big(\sum_{i=1}^kx_iA_i\Big)\Bigg)\label{mg}.
	\end{align}
	The coefficient of $x_2x_3 ... x_k$ in (\ref{kg}) is given by $k!\int_{\mathbb{S}^{n}}\prod_{i=1}^k<A_i\xi,\xi>d\sigma(\xi)$. Identifying this with the coefficient of $x_2x_3 ... x_k$ in (\ref{mg}) provides an explicit value for (\ref{fgh}). Notice that if   the explicit value of (\ref{fgh}) is given for a certain $k_0$,  then one can directly obtain (\ref{fgh}) for all $k\in{1,2,\cdots,k_0}$. Moreover, given the explicit value  of (\ref{fgh}) at $l=2k$ then  choosing $A_1=A_2=\cdots=A_k=A$ and $A_{k+1}=A_{k+2}=\cdots=A_{2k}=A^\star$ provides the exact value of $\Phi'_{2k}(A)$.
	\bibliographystyle{plainnat}

\begin{thebibliography}{20}{
			\bibitem{agbor} D. Agbor and W. Bauer, Heat flow and an algebra of Toeplitz operators, \emph{Integral Equations and Operator theory}  70, (2015), pp. 271-299.			
			\bibitem{axler}S. Axler, P. Bourdon and R. Wade, Harmonic function theory, \emph{Springer Science and Business Media}, Vol. 137 Graduate texts in Mathematics, 2013.
			\bibitem{aubrun} G. Aubrun and I. Nechita, The multiplicative property characterizes $l_p$ and $L_p$ norms, \emph{Conflu. Math.}, Vol. 03 No. 04 (2011), pp. 637–647.
			\bibitem {bauer1} W. Bauer, Berezin-Toeplitz quantization and composition formulas, \emph{Journal of  Functional Analysis} 256, (2009), pp.  3107-3142.
			\bibitem{bauer} W. Bauer and H. Issa, Commuting Toeplitz operators with quasi-homogeneous symbols on the Segal-Bargmann space, \emph{J. Math. Anal. Appl.} 386 (2012), pp. 213-235.
			\bibitem{bhatia} R. Bhatia, Matrix analysis, \emph{Springer Science and Business Media,}, 1996.
		\bibitem{b1}	R. Bhatia, Some inequalities for norm ideals, \emph{Commun. Math. Phys.} 111 (1987),  pp. 33-39.
	\bibitem{b2}	R. Bhatia, Perturbation inequalities for the absolute value map in norm ideals of operators, \emph{J. Operator Theory} 19 (1988), pp. 129-136.
\bibitem{b3} R. Bhatia and C. Davis, Relations of linking and duality between symmetric gauge functions, \emph{Operator Theory: Advances and Applications}, 73 (1994) pp. 127-137.
			\bibitem{hol} R. Bhatia and J .A.R. Holbrook, Unitary invariance and spectral variation,	\emph{ Linear Algebra Appl.,} 95 (1987), pp. 43-68.
			\bibitem{chu}	Y.-M.Chu, G.-D.Wang and X.-H. Zhang, The Schur multiplicative and harmonic convexities of the complete symmetric function. \emph{Math. Nachr}. 284 (5–6) (2011), pp. 653-663.
			\bibitem {coburn}L. A. Coburn, On the Berezin-Toeplitz calculus, \emph{ Proc. Am. Math. Soc.} 129 (11), (2007), pp.  3331-3338.
			\bibitem{cuttler} A. Cuttler, C. Greene and M. Skandera, Inequalities for symmetric means, \emph{European Journal of Combinatorics}  Vol.32, I 6  (2011), pp. 745-761.
			\bibitem{foll} G.B. Folland, How to integrate a polynomial over a sphere, \emph{American Mathematical Society}, Vol. 108 No. 5 (2001), pp. 446-448.
			\bibitem{folland} G.B. Folland, Real analysis, \emph{John Wiley, New York}, 1999.
			\bibitem{fong} C.-K. Fong and J. A.R. Holbrook, Unitarily invariant operator norms, \emph{Canad. J. Math.}, 35 ( 1983) pp. 274-299.
				\bibitem{guan}	K. Guan, Schur-convexity of the complete elementary symmetric function, \emph{J. Inequal. Appl.},  (2006) pp.1-9.
			\bibitem{hunter} D.B. Hunter,  The positive-definiteness of the complete symmetric functions of even order. Math. \emph{Proc. Camb. Philos. Soc.} 82 (2) (1977), pp. 255--258.
			\bibitem{harver} D.A. Harver, Matrix algebra from a statistician's perspective, \emph{Springer Science and Business Media}, 2008.
			\bibitem{hndrsn} H.V. Henderson and S.R. Searle, S.R., The vec-permutation, the vec-operator and Kronecker products: a review. \emph{Linear Multilinear Algebra} 9 (1981), pp. 271--288.
			\bibitem{Issa} H. Issa, Construction of some classes of commutative Banach and $C^\star$-algebras of Toeplitz operators,\emph{ C. R. Acad. Sci. Paris}, Ser 1 357 (2019), pp. 389-394.
			\bibitem{Issa1} H. Issa, The analysis of Toeplitz operators, commutative Toeplitz algebras and applications to heat kernel constructions, \emph{S.U.B G\"{o}ttingen}, 2012.
\bibitem{IAM}  H, Issa, H. Abbas, B. Mourad, On an integral representation of the normalized trace of the k-th symmetric tensor power of matrices and some applications, arXiv:2106.01486, ((2021).
			\bibitem{kania} T. Kania, A short proof of the fact that the matrix trace is the expectation of the numerical values, \emph{ American Mathematical Monthly}, 122 Vol. 8 (2015), pp. 782-783.
			\bibitem{krantz} S. G. Krantz, Function theory of several complex variables, \emph{AMS Chelsea Publishing Series, Amer. Math. Soc.}, 2001.
			\bibitem{Li}Chi-K.   Li  and  N. K.  Tsing, Norms that are invariant under unitary similarities and the C-numerical radii, \emph{Linear and Multilinear Algebra,} 24 (1989), pp. 209-222.
			\bibitem{loewy} R. Loewy and D. London, A note on the inverse problem for nonnegative matrices,\emph{ Linear Multilinear Algebra} 6 (1978), pp. 83--90.
			\bibitem{mcd}I. Macdonald, Symmetric functions and Hall polynomials, \emph{Oxford science publications: Oxford mathematical monographs}, 1998.
			\bibitem{marcus}  M. Marcus and L. Lopes, Inequalities for symmetric functions and Hermitian matrices,\emph{ Canad. J. Math.} 9 (1957), pp. 305-312.
			\bibitem{Marshall} A. W. Marshall, I. Olkin and B. C. Arnold, Inequalities: Theory of majorization and its applications, \emph{Springer-Verlag New York,}  second edition,  2011.
			\bibitem{mcleod}J. B. McLeod, On four inequalities in symmetric functions,\emph{ Proceedings of the Edinburgh Mathematical Society,} 11-4 (1959), pp. 211-219.
	\bibitem{menon}		K.V.Menon ,Inequalities for symmetric functions, \emph{ Duke Mathematical Journal} 35 (1968), pp. 3-45.
			\bibitem{mirsky}L. Mirsky, Symmetric gauge functions and unitarily invariant norms,	\emph{Quart . J. Math. , Oxford,} Ser. 2, 11 (1960), pp. 50-59.
			\bibitem{nic}		C. P. Niculescu, Convexity according to the geometric means, \emph{Math. Inequal. Appl.}  3(2) (2000), pp. 155--167.
			\bibitem{neu} J. Von Neumann, Some matrix-inequalities
			and metrization of matric-space, \emph{Tomsk Univ. Rev.} 1 (1937) pp. 286-300
			\bibitem{roventa} I. Roventa, and L. E. Temereanca,  A note on the positivity of the even degree complete homogeneous symmetric polynomials, \emph{Mediterr. J. Math.}    (2019),  pp. 1-16.
			\bibitem{rudin} W. Rudin, Function theory in the unit ball of $\mathbb{C}^n$, \emph{A series of comprehensive studies in Mathematics, Springer}, 2008.
			\bibitem{saitoh} S. Saitoh,  Theory of reproducing kernels and its applications, \emph{Pitman research notes in mathematics series, Longman Scientific \& Technical}, 1988.
			\bibitem{sra} S.  Sra, New concavity and convexity results for symmetric polynomials and their ratios,  \emph{Linear and Multilinear Algebra,} 68 Vol 5 (2020),  pp. 1031-1038.
			\bibitem{thompson} R.C. Thompson, Convex and concave functions of singular values of matrix
			sums, \emph{Pacific J. Math.} , 66 (1976), pp. 285-290.
			\bibitem{zhu} K. Zhu, Operator theory in function spaces, \emph{Mathematical surveys and monographs, Amer. Math. Soc.}, 2007.
					}
	\end{thebibliography}
	
	\appendix
		\section{Appendix}\label{a1}
	The appendix is devoted to find the explicit coefficient for the $xyz$-variable in the multinomial (\ref{i2}).
	\begin{lemma}\label{laa}
		The coefficent of the $xyz$-variable in
		\begin{equation}
		\begin{split}
		6tr\big[P^4(x,y,z)\big]&+8tr\big[P^3(x,y,z)\big].tr\big[P(x,y,z)\big]+3\bigg[tr \big[P^2(x,y,z)\big]\bigg]^2\\
		&+6tr \big[P^2(x,y,z)\big]\cdot\bigg[tr \big[P(x,y,z)\big]\bigg]^2+\bigg[tr \big[P(x,y,z)\big]\bigg]^4.
		\end{split}
		\end{equation}	
		is given by
		\begin{fleqn}[\parindent]
			\begin{equation}\label{aa001}
			\begin{split}
			24\Bigg\{tr&\Big[ABCD+ABDC+ACBD+ACDB+ADBC+ADCB\Big] +tr(A)tr \Big[BCD+BDC\Big]\\
			&+tr(B)tr \Big[ACD+ADC\Big]+tr(C)tr \Big[ABD+ADB\Big]+tr(D)tr \Big[ABC+ACB\Big]\\
			&+tr(AB)\Big[tr (CD)+tr(C)tr(D)\Big]+tr(AC)\Big[tr (BD)+tr(B)tr(D)\Big]\\
			&+tr(AD)\Big[tr (BC)+tr(B)tr(C)\Big]+tr(A)\Big[tr(B)tr (CD)+tr(C)tr (BD)+tr(D)tr (BC)\Big]\\
			&+tr(A)tr(B)tr(C)tr(D) \Bigg\}.
			\end{split}
			\end{equation}
		\end{fleqn}	
	\end{lemma}
	The proof is achieved by computing the exact coefficient for the $xyz$-variable in each of the five terms of (\ref{laa}) and this is presented in the next propositions.	
		\begin{proposition}
		The coefficient of $xyz$ in $P^4(x,y,z)$ is given by
		\begin{align}\label{i4}
		(A&B+BA)(CD+DC)+(AC+CA)(BD+DB)+(AD+DA)(BC+CB)\notag\\
		&+(BC+CB)(AD+DA)+(BD+DB)(AC+CA)+(CD+DC)(AB+BA).
		\end{align}
		In particular, the $xyz$-coefficient of $tr\big[P^4(x,y,z)\big]$ is given by
		\begin{equation}\label{aa002}4tr\Big[ABCD+ABDC+ACBD+ACDB+ADBC+ADCB\Big].\end{equation}
	\end{proposition}
	\begin{proof}
		The matrix $P^2(x,y,z)$ is given by
			\begin{align}\label{opp}
		P^2(x,y,z)&=\big(A+xB+yC+zD\big)\big(A+xB+yC+zD\big)\notag\\
		&=A^2+(AB+BA)x+(AC+CA)y+(AD+DA)z+(BC+CB)xy\notag\\
		&\quad+(BD+DB)xz+(CD+DC)yz+x^2B^2+y^2C^2+z^2D^2.
		\end{align}
			Hence, the $xyz$-coefficient in $P^4(x,y,z)$ is  the $xyz$-coefficient in
		\begin{align*}
		\Big[(AB+BA)x+(AC+CA)y+(AD+DA)z&+(BC+CB)xy+(BD+DB)xz+(CD+DC)yz\Big]\\
		&\times \\
		\Big[(AB+BA)x+(AC+CA)y+(AD+DA)z&+(BC+CB)xy+(BD+DB)xz+(CD+DC)yz\Big].
		\end{align*}
		from which (\ref{i4}) is obtained. Hence the $xyz$-coefficient in $tr\big[P^4(x,y,z)\big]$ is given by
		\begin{align*}
		2T&r\Big[(AB+BA)(CD+DC)+(AC+CA)(BD+DB)+(AD+DA)(BC+CB)\Big]\\
		=2T&r\Big[ABCD+ABDC+BACD+BADC+ACBD+ACDB+CABD+CADB\\
		&\quad\quad\quad\quad\quad\quad\quad\quad\quad\quad\quad\quad\quad\quad\quad\quad\quad\quad\quad\quad\quad+ADBC+ADCB+DABC+DACB\Big]\\
		=4T&r\Big[ABCD+ABDC+ACBD+ACDB+ADBC+ADCB\Big].
		\end{align*}
	\end{proof}
		\begin{proposition}
		The coefficient of $xyz$ in $tr\big[P^3(x,y,z)\big].tr\big[P(x,y,z)\big]$ is given by
		\begin{equation}\label{aa003}
		3\Bigg\{tr(A)tr \Big[BCD+BDC\Big]+tr(B)tr \Big[ACD+ADC\Big]+tr(C)tr \Big[ABD+ADB\Big]+tr(D)tr \Big[ABC+ACB\Big]\Bigg\}.
		\end{equation}
	\end{proposition}
	\begin{proof}
		By (\ref{opp}), we know that  \begin{align}\label{i5}
		P^3(x,y,z)=\Big[A^2&+(AB+BA)x+(AC+CA)y+(AD+DA)z+(BC+CB)xy+(BD+DB)xz\notag\\
		&+(CD+DC)yz+x^2B^2+y^2C^2+z^2D^2\Big] \times\Big[A+xB+yC+zD\Big].
		\end{align}
		This shows that in $P^3(x,y,z)$ the coefficient of the
		\begin{enumerate}
				\item $xyz$-term is given by
			$(BC+CB)D+(BD+DB)C+(CD+DC)B$ and whose trace is given by
			\begin{equation}\label{i5}
			tr\Big[BCD+CBD+BDC+DBC+CDB+DCB\Big]=3tr\Big[BCD+BDC\Big]
			\end{equation}
			\item $xy$-term is given by $(AB+BA)C +(AC+CA)B +(BC+CB)A$ and whose trace is given by
			\begin{equation}\label{i6}
			tr\Big[ABC+BAC+ACB+CAB+BCA+CBA\Big]=3tr\Big[ABC+ACB\Big]
			\end{equation}
			\item $yz$-term is given by $(AC+CA)D +(AD+DA)C +(CD+DC)A$ and whose trace is given by
			\begin{equation}\label{i7}
			tr\Big[ACD+CAD+ADC+DAC+CDA+DCA\Big]=3tr\Big[ACD+ADC\Big]
			\end{equation}
			\item $xz$-term is given by $(AB+BA)D +(AD+DA)B +(BD+DB)A$ and whose trace is given by
			\begin{equation}\label{i8}
			tr\Big[ABD+BAD+ADB+DAB+BDA+DBA\Big]=3tr\Big[ABD+ADB\Big].
			\end{equation}
		\end{enumerate}
		To get the  coefficient of $xyz$ in $tr\big[P^3(x,y,z)\big].tr\big[P(x,y,z)\big]$, we multiply (\ref{i5}) by $tr(A)$,  (\ref{i6}) by $tr(D)$, (\ref{i7}) by $tr(B)$ and (\ref{i8}) by $tr(C)$ from which the proposition then follows.
	\end{proof}
	\begin{proposition}
		The coefficient of $xyz$ in $\bigg[tr \big[P^2(x,y,z)\big]\bigg]^2$ is given by
		\begin{equation}\label{i9}
		8\Big\{tr(AB)tr(CD)+tr(AC)tr(BD)+tr(AD)tr(BC) \Big\}
		\end{equation}
		and the $xyz$-coefficient in $tr \big[P^2(x,y,z)\big]\cdot\bigg[tr \big[P(x,y,z)\big]\bigg]^2$ is given by
		\begin{align}\label{i10}
		4\Big\{tr(AB)tr(C)tr(D)&+tr(AC)tr(B)tr(D)+tr(AD)tr(B)tr(C)\notag\\
		&+tr (CD)tr(A)tr(B)+tr (BD)tr(A)tr(C)+tr (BC)tr(A)tr(D)\Big\}.
		\end{align}	
	\end{proposition}
	\begin{proof}
		Following (\ref{opp}) , we have
		\begin{align}\label{i11}
		tr \big[P^2(x,y,z)\big]=2\big[&xtr(AB)+ytr(AC)+ztr(AD)+xytr(BC)+xztr(BD)+yztr(CD)\Big]\\
		&+tr(A^2+x^2B^2+y^2C^2+z^2D^2).\notag
		\end{align}
		Squaring the preceding equality, the $xyz$ coefficient is equal to
		\begin{align*}4\big[tr(AB)tr(CD)+tr(AC)tr(BD)+tr(AD)tr(BC)+tr(BC)tr(AD)+tr(B&D)tr(AC)\\
		&+tr(CD)tr(AB) \big].
		\end{align*}
		for which (\ref{i9}) follows. Moreover,   $\bigg[tr \big[P(x,y,z)\big]\bigg]^2$ is given by
		\begin{align*}
		2\Big[xtr(A)tr(B)+ytr(A)tr(C)+ztr(A)tr(D)&+xytr(B)tr(C)+xztr(B)tr(D)+yztr(C)tr(D)\Big]\\
		&+(trA)^2+x^2(trB)^2+y^2(trC)^2+z^2(trD)^2.\notag
		\end{align*}
		Multiplying the above equation by (\ref{i11}) we get (\ref{i10}).
			\end{proof}
	Finally, we note that $$\bigg[tr \big[P(x,y,z)\big]\bigg]^4=\bigg[tr(A)+xtr(B)+yr(C)+ztr(D)\bigg]^4=\sum_{\substack{|\alpha|=4 \\ \alpha\in\mathbb{N}_0^4}}\binom{4}{\alpha}\Big(tr(A),xtr(B),ytr(C),ztr(D)\Big)^\alpha.$$
		So that for $\alpha=(1,1,1,1)$ we obtain the coefficient of the $xyz$-variable:
	\begin{equation}\label{cvt}
	24tr(A)tr(B)tr(C)tr(D).
	\end{equation}
	Multiplying (\ref{aa002}) by 6, (\ref{aa003}) by 8, (\ref{i9}) by 3, (\ref{i10}) by 6 and then adding all together to (\ref{cvt}) we get (\ref{aa001}).
	\end{document}